\documentclass[12pt,reqno]{amsart}
\usepackage{amsmath,amssymb,amsfonts,mathrsfs,mathtools,enumitem}
\usepackage[all]{xy}
\usepackage{hyperref}

\usepackage[all]{xy}
\usepackage{hyperref}

\usepackage{cleveref}

\hypersetup{
	colorlinks,
	citecolor=red,
	filecolor=black,
	linkcolor=blue,
	urlcolor=black
}

\theoremstyle{plain}
\newtheorem{thm}{Theorem}[section]
\newtheorem{con}[thm]{Conjecture}
\newtheorem{cor}[thm]{Corollary}
\newtheorem{lmm}[thm]{Lemma}
\newtheorem{prp}[thm]{Proposition}

\theoremstyle{definition}
\newtheorem{defn}[thm]{Definition}
\newtheorem{rem}[thm]{Remark}
\newtheorem{convention}[thm]{Convention}

\numberwithin{equation}{section}

\begin{document}
\title[Injective self-maps of algebraic varieties]{On the injective self-maps of algebraic varieties}

\author[I. Biswas]{Indranil Biswas}

\address{Mathematics Department, Shiv Nadar University, NH91, Tehsil Dadri, Greater Noida, Uttar Pradesh 
201314, India}

\email{indranil.biswas@snu.edu.in, indranil29@gmail.com}

\author[N. Das]{Nilkantha Das}

\address{Stat-Math Unit, Indian Statistical Institute, 203 B.T. Road, Kolkata 700 108, India}

\email{dasnilkantha17@gmail.com}

\subjclass[2020]{14E25, 14R10}

\keywords{Endomorphism of varieties; K\"{a}hler differentials; Miyanishi conjecture}

\begin{abstract}
A conjecture of Miyanishi says that an endomorphism of an algebraic variety, defined over an algebraically closed field of characteristic zero,
is an automorphism if the endomorphism is injective outside a closed subset of codimension at least $2$. We prove the conjecture in the following cases:
\begin{enumerate}
\item The variety is non-singular.

\item The variety is a surface.

\item The variety is locally a complete intersection that is regular in codimension $2$.
\end{enumerate}
We also discuss a few instances where an endomorphism of a variety, satisfying the hypothesis of the
conjecture of Miyanishi, induces an automorphism of the non-singular locus of the variety. Under additional
hypotheses, we prove that the conjecture holds when the variety has only isolated singularities.
\end{abstract}

\maketitle

\section{Introduction}

Throughout, $k$ denotes an algebraically closed field of characteristic zero unless specified otherwise.

\subsection*{Theorem of Ax} A famous theorem of Ax, \cite{Ax}, and Nowak says that an injective endomorphism of an algebraic variety $X$ over 
$k$ is an automorphism. Strictly speaking, Ax proved that an injective endomorphism of $X$ is necessarily 
bijective, and later Nowak, \cite{Nowak}, showed that a bijective endomorphism of $X$ is necessarily an 
automorphism.

Since the initial work of Ax, this theorem has been proved via different approaches (see \cite{Borel}, \cite{Iitaka} 
for example). This result was preceded by a similar result for injective endomorphisms of 
$\mathbb{A}^n_{\mathbb{R}}$ by Bia{\l}ynicki-Birula and Rosenlicht \cite{Bialynici-Birula-Rosenlicht}. The 
result of \cite{Bialynici-Birula-Rosenlicht} was later extended by Kurdyka and Rusek, \cite{Kurdyka-Rusek}, to 
semi-algebraic maps.

Borel's proof in \cite{Borel} of the theorem of \cite{Ax} also applies to non-singular varieties defined over the
field of real numbers. Further generalizations were found by Cynk and Rusek \cite{Cynk-Rusek}. Subsequently, Kurdyka, 
\cite{Kurdyka}, extended the result of Borel to singular real algebraic varieties under the weakened condition 
that the injective morphism is not necessarily invertible, but it is at least surjective. Such a morphism is, in fact, a 
homeomorphism, as shown by Parusi\'nski \cite{Parusinski}. Grothendieck, \cite{EGA-IV}, proved an analog of 
Ax's theorem for relative schemes and projective limits of schemes. Gromov in \cite{Gromov} investigated the 
injective endomorphisms of symbolic varieties.

\subsection*{Miyanishi conjecture} The significance of the above theorem of Ax is that injectivity, which is a set-theoretic property of an 
endomorphism, enforces automorphism which is an algebraic property. It is natural to ask if the global 
injectivity hypothesis in the theorem of Ax can be weakened. Along this direction, M. Miyanishi \cite{Miyanishi} 
in 2005 proposed a conjecture which is referred to as the Miyanishi conjecture.

\begin{con}[Miyanishi]\label{main_conj}
Let $\phi\,:\,X \,\longrightarrow\, X$ be an endomorphism of an algebraic variety $X$ over an algebraically
closed field of characteristic zero, and let $Y\, \subset\, X$ be a proper closed algebraic subset
such that the restriction of $\phi$ to the complement
$X \setminus Y$ is injective. Furthermore suppose that either $\phi$ is quasi-finite
or the codimension of $Y$ in $X$ is at least $2$. Then $\phi$ is an automorphism.
\end{con} 

If $\phi$ is quasi-finite, then Conjecture \ref{main_conj} holds as was shown by the second author
\cite[Theorem 1.2]{Das}. Let us denote the codimension of a subvariety $Y$ in $X$ by $\mathrm{codim}_XY$ for
notational convenience. In 
the case where $\mathrm{codim}_X Y \,\geq\, 2$, Kaliman in \cite{Kaliman} showed that the Miyanishi conjecture
holds if the variety $X$ is either affine or complete. Several other cases were considered by the
second author, \cite[Theorem 1.5 and Corollary 1.6]{Das}, with $\mathrm{codim}_XY \,\geq\, 2$. More precisely, it
was shown there that 
the Miyanishi conjecture holds if either $\phi$ is affine, or $X$ has a normalization which is 
locally factorial, or $k\,=\,\mathbb{C}$ with $X$ being a Stein algebraic variety. 
In particular, the Miyanishi conjecture is valid if $X$ is non-singular.

The goal of this article is to investigate the conjecture in some other cases. To prove \Cref{main_conj},
it suffices to establish it for normal varieties (see \Cref{normal_case_reducing}). In view of this, we stick to the case of normal varieties. 

A variety is said to be locally a complete intersection if the local ring at each point is the quotient of a regular local ring by an ideal generated by a regular sequence.

\begin{thm}\label{main_thm_codim_case}
Let $X$ be a normal algebraic variety defined over an algebraically closed field of characteristic zero, and
let $\phi$ be an endomorphism of $X$. Let $Y$ be a proper closed algebraic subset of $X$ of codimension
at least $2$ such that the restriction $$\phi\big\vert_{X \setminus Y}\ : \ X \setminus Y \ \longrightarrow \ X$$ is injective.
Assume that at least one of the following statements holds:
\begin{enumerate}[label=(\roman*)]
\item $X$ is an algebraic surface.
\item $X$ is non-singular.
\item $X$ is locally a complete intersection and it is regular in codimension $2$.
\end{enumerate}
Then $\phi$ is an automorphism.
\end{thm}

As discussed above, the Miyanishi conjecture is known to be true for non-singular varieties (see 
\cite[Corollary 1.6]{Das}). But the proof in \cite{Das} uses analytic methods. In contrast, all the results
here are proved algebro-geometrically.

\subsection*{Endomorphisms restricted to non-singular locus}
Next, we consider a necessary condition for the Miyanishi conjecture. If $\phi$, in \Cref{main_conj}, is an 
automorphism, its restriction to the non-singular locus of $X$ is evidently an automorphism. The following
result investigates this necessary condition under some hypotheses.

\begin{thm}\label{Automorphism of non-singular locus theorem}
Let $X$ be a normal algebraic variety defined over an algebraically closed field of characteristic zero, and
let $\phi$ be an endomorphism of $X$. Let $Y$ be a proper closed algebraic subset of $X$ of codimension at
least $c$ such that the restriction $$\phi\big\vert_{X \setminus Y}\ : \ X \setminus Y \
\longrightarrow \ X$$ is injective. Denote the non-singular locus of $X$ by $W$. Assume that at least one of the following statements holds:
\begin{enumerate}[label=(\roman*)]
\item $c \ = \ 2$, and $\phi$ is surjective.
\item $c \ = \ 2$, and $X$ is a $\mathbb{Q}$--factorial variety.
\item $c \ = \ 2$, and $X$ is locally a complete intersection.
\item $c \ = \ 3$.
\end{enumerate}
Then $\phi\ (W) \ \subseteq \ W$, and the restriction $$\phi\big\vert_{W}\ : \ W \ \longrightarrow \ W$$ is an automorphism of $W$.
\end{thm}

Assume that the endomorphism $\phi$ of $X$ restricts to an automorphism of
the smooth locus of $X$. Under this assumption, we show that the Miyanishi conjecture holds for varieties with only isolated singularities.

\begin{thm}\label{isolated singularity theorem}
Let $X$ be a normal algebraic variety, defined over an algebraically closed field of characteristic zero,
which has only finitely many singular points, and 
let $\phi$ be an endomorphism of $X$. Let $Y$ be a proper closed algebraic subset of $X$ of codimension
at least $2$ such that the restriction $$\phi\big\vert_{X \setminus Y}\ : \ X \setminus Y \ \longrightarrow \ X$$ is injective. Denote the
non-singular locus of $X$ by $W$. 
Assume that the restriction of $\phi$ to $W$ induces an automorphism of $W$. Then $\phi$ is an automorphism.
\end{thm}

Theorem \ref{Automorphism of non-singular locus theorem} and Theorem \ref{isolated singularity theorem}
together imply the following:

\begin{cor}\label{cor1}
Let $X$ be a normal algebraic variety, defined over an algebraically closed field of characteristic zero,
which has only finitely many singular points, and let $\phi$ be an endomorphism of $X$.
Let $Y$ be a proper closed algebraic subset of $X$ of codimension
at least $2$ such that the restriction $$\phi\big\vert_{X \setminus Y}\ : \ X \setminus Y \ \longrightarrow \ X$$ is injective.
Assume that at least one of the four conditions in Theorem \ref{Automorphism of non-singular locus theorem}
holds. Then $\phi$ is an automorphism.
\end{cor}

The structure of the paper is as follows. In \Cref{section prilim}, we begin by recalling and establishing a 
few general results concerning endomorphisms of varieties, which will be utilized in the subsequent sections. 
The first two cases of Theorem \ref{main_thm_codim_case} are addressed in \Cref{section first two case} 
(\Cref{th1}). \Cref{Automorphism of non-singular locus theorem} and \Cref{isolated singularity theorem} are 
presented in \Cref{section four} (\Cref{isolated sing section case} and \Cref{automorphism restriction 
result}). The discussions in \Cref{section fourth case} primarily involve homological algebra related to 
various K\"{a}hler differential sheaves. The third case of \Cref{main_thm_codim_case} is addressed in \Cref{section 
fourth case} (\Cref{th3}).

\noindent\textbf{Notation:}\,\, By an algebraic variety over $k$, we mean an integral separated scheme $X$ 
of finite type over $k$. The structure sheaf of $X$ is denoted by $\mathcal{O}_X$. For a sheaf $\mathcal{F}$ 
on a topological space $X$, the ring associated to an open subset $U$ of $X$ is denoted by $\Gamma\left( 
U,\,\mathcal{F}\right)$; the stalk of $\mathcal F$ at a point $x\,\in\, X$ is denoted by $\mathcal{F}_x$. We 
denote the relative K\"{a}hler differential sheaf of a regular morphism $X \ \longrightarrow \ Y$ by 
$\Omega_{X/Y}$, and the K\"{a}hler differential sheaf of an algebraic variety $X/k$ is denoted by $\Omega_X$. 
All the rings are assumed to be commutative Noetherian with unity, and all the modules are assumed to be 
finitely generated.

\section{Preliminaries}\label{section prilim} 

We will put down a few results that will be useful in the proof of the main theorem.
The first step is to reduce the Miyanishi conjecture to the case of normal varieties. 

\begin{lmm}\label{normal_case_reducing}
If the Miyanishi conjecture is valid for normal algebraic varieties, then it is valid in full generality.
\end{lmm}

The reader is referred to \cite[Lemma 2]{Kaliman} for a proof of Lemma \ref{normal_case_reducing} (see also \cite[Lemma 2.2]{Das}).

In view of Lemma \ref{normal_case_reducing}, we can assume $X$ to be a normal variety.

\begin{rem}\label{rem1}
Let $X$ be a normal variety, and let $\phi$ be an endomorphism such that the restriction
$$\phi_1\ :=\ \phi\big\vert_{X \setminus Y}\ :\ X \setminus Y \ \longrightarrow\ X$$ is injective
for some proper closed subset $Y$ of $X$. 
Therefore $\phi_1$ is birational. By Zariski's Main Theorem (cf. \cite[Chapter III.9]{Red_book}), $\phi_1$ is an open immersion. Hence the subset ${\rm Image}(\phi_1)\, \subset\, X$ is also open, and
the map
$$
\phi_1\ :\ X \setminus Y \ \longrightarrow\ {\rm Image}(\phi_1)$$
is an isomorphism.
\end{rem}

Some common notation will be used several times. Let us first make some notational conventions which will be 
used throughout the article.

\begin{convention}\label{conv1}
Let $X$ be a normal algebraic variety defined over an algebraically closed field $k$, and
let $Y$ be a closed algebraic subset of $X$. Let $$\phi\ :\ X \ \longrightarrow\ X$$
be an endomorphism of $X$ such that the restriction
$$
\phi_1\ :=\ \phi\big\vert_{X \setminus Y}\ :\ X \setminus Y \ \longrightarrow\ X$$
is an injective map. We denote the open subsets $X \setminus Y$ and $\phi (X \setminus Y)\,=\, {\rm Image}
(\phi_1)$ by $U$ and $V$, respectively. The closed subset $X \setminus {\rm Image}(\phi_1)$ is denoted by
$Z$. The singular locus of $X$ and the non-singular locus of $X$ are denoted by $S$ and $W$ respectively.
\end{convention}

Assuming Convention \ref{conv1}, we have seen in Remark \ref{rem1} that $\phi_1$ is an isomorphism
between $U$ and $V$. The next result determines the inverse image of $V$ under the map $\phi$.

\begin{lmm}\label{preimage_preserving_lemma}
Let $\phi$ be an endomorphism of $X$, and use Convention \ref{conv1}. If there are two open subsets
$U$ and $V$ of $X$ such that the restriction map
$$\phi_1\ :=\ \phi\big\vert_{U}\ :\ U \ \longrightarrow\ V$$
is an isomorphism, then $\phi^{-1}(V)\,=\,U$.
\end{lmm}

\begin{proof}
Note that $U$ is an open subset of $\phi^{-1}(V)$. Let $\iota \, :\, U \,\hookrightarrow \,
\phi^{-1}(V)$ be the open immersion, and let
$$\phi_2\, :=\, \phi\big\vert_{\phi^{-1}(V)}\, :\, \phi^{-1}(V) \, \longrightarrow \, V$$ be the restriction of $\phi$. Thus
we have $\phi_2 \circ \iota \,= \,\phi_1$. It
follows immediately from \cite[Corollary II.4.8]{Hartshorne} that the map $\iota$ is proper. Being an open
immersion, $\iota$ can be proper only if $U$ is a connected component of $\phi^{-1}(V)$. On the other hand,
$\phi^{-1}(V)$ is irreducible, because $\phi^{-1}(V)$ an open subset of $X$ and $X$ is irreducible. In
particular, $\phi^{-1}(V)$ is connected. It now follows that $U$ and $\phi^{-1}(V)$ coincide.
\end{proof}

\begin{rem}
In the proof of Lemma \ref{preimage_preserving_lemma} it was used that the morphism $\phi_2\,:\,
\phi^{-1}(V) \,\longrightarrow \,V$ is separated. Separateness is, in fact, necessary for the Miyanishi
conjecture to hold. An example showing this is discussed in \cite[Example 3.6]{Das} and \cite{Kaliman}.
\end{rem}

Next, we recall a result of Jarden which says that the two closed algebraic subsets 
$Y$ and $Z$ of $X$ in Convention \ref{conv1} have the same dimension (see 
\cite[Theorem 1]{Jarden}).

\begin{prp}[{\cite[Theorem 1]{Jarden}}]\label{equal_dim_lemma}
Let $X$ be an algebraic variety, and let $\phi\, :\, X\, - - \rightarrow\, X$ be a rational map.
Let $Y,\, Z$ be two closed sub-varieties of $X$ such that $\phi$ is defined on $X\setminus Y$, and
the restriction $\phi\big\vert_{X \setminus Y}$ is a bijection of $X \setminus Y$
with $X \setminus Z$. Then the equality $\dim Y \, =\, \dim Z$ holds.
\end{prp}

We note that Kaliman proved \Cref{equal_dim_lemma} using algebraic de Rham homology under the additional 
assumptions that $\phi$ is regular and $\mathrm{codim }_XY\, \geq \, 2$ (see \cite[Lemma 3]{Kaliman}). 
See \cite[Lemma 10]{Lamy_Sebag} for an algebraic proof of Proposition \ref{equal_dim_lemma}.

While proving \Cref{main_thm_codim_case}, we show that $\phi$ in \Cref{main_thm_codim_case} is quasi-finite in 
each of the cases, and then we show that it is an automorphism. To avoid repetition of the arguments every 
time, the following lemma is formulated.

\begin{lmm}\label{quasi-finite to automorphism}
Let $X, \, Y$ and $\phi$ be as in Convention \ref{conv1}. If $\phi$ is quasi-finite, then it is an automorphism.
\end{lmm}

\begin{proof}
In \Cref{rem1}, it was observed that $\phi$ is birational. Given that $\phi$ is quasi-finite,
by Zariski's Main Theorem (cf. \cite[Chapter III.9]{Red_book}), the map $\phi$ is an open immersion.
Thus $\phi$ is injective. Hence
it is an automorphism by the theorem of Ax.
\end{proof}

The next proposition will be useful in the proof of the main theorem.

\begin{prp}\label{injectivity of stalk maps}
Let $X, \, Y$ and $\phi$ be as in Convention \ref{conv1}, with $\mathrm{codim}_XY\, \geq \, 2$. Then for
each closed point $x \,\in\, X$, the induced ring homomorphism
\begin{align}\label{stalk_map}
\phi_x\ :\ \mathcal{O}_{X,\, \phi\,(x)} \ \longrightarrow \ \mathcal{O}_{X,\,x}
\end{align}
 is injective.
\end{prp}

\begin{proof}
We follow earlier notational conventions. Given an open neighborhood $U$ of $x$ and an element
$f\, \in \, \Gamma\, (U, \, \mathcal{O}_X \,)$, the germ of the function $f$ will be denoted by
$[(\, U,\, f\,)] \, \in\, \mathcal{O}_{X,\, x}$.

Let $A$ be an open subset of $X$ containing $\phi\,(x)$, and let $f$ be a regular function on $A$ such that 
the germ $[(\, \phi^{-1}\,(A),\, f \circ \phi\,)] \, \in\, \mathcal{O}_{X,\, x}$ vanishes, meaning there is 
an open neighborhood $B \, \subset \, \phi^{-1}\,(A)$ of $x$ such that $f \, \circ\, \phi \, =\,0$ on $B$. In 
particular, $f \, \circ\, \phi \, =\,0$ on $B\, \cap \, U$. Since $\phi_1\,:=\,
\phi\big\vert_U \, :\, U \, \longrightarrow \, V$ is an isomorphism, it follows that
$f\,=\,0$ on $\phi(B \, \cap \,U)$. The variety $X$ being irreducible, $\phi(B \, \cap \,U)$ 
is a dense open subset of $A$. Thus $f\,=\,0$ on $A$, proving the proposition.
\end{proof}

\subsection*{Homomorphisms of coherent sheaves}

A coherent sheaf $\mathcal{F}$ on $X$ is said to be torsion-free (respectively, reflexive) if the natural map 
$$\mathcal{F}\ \longrightarrow \ \mathcal{F}^{\vee \vee}$$ is injective (respectively, bijective). Here 
$\mathcal{V}^{\vee}$ denotes the $\mathcal{O}_X$--dual of $\mathcal{V}$. The following proposition is standard; see \cite[Proposition 1.4.1]{Bruns-Herzog}. It will be useful in the proof of the main result.

\begin{prp}\label{isomorphism_of_vector_bundles_lemma}
Let $X$ be a normal algebraic variety over $k$, and let $\alpha\,:\, \mathcal{F} \,\longrightarrow
\,\mathcal{G}$ be a homomorphism of coherent sheaves on $X$ whose restriction over $X \setminus Y$
is injective, where $Y\, \subset\, X$ is some closed algebraic subset of codimension at least $2$.
Assume that $\mathcal{G}$ is torsion-free. Then the following two statements hold:
\begin{enumerate}
\item The homomorphism $\alpha$ is injective if $\mathcal{F}$ is torsion-free.

\item If $\mathcal{F}$ is reflexive and the restriction of $\alpha$ to $X\setminus Y$
$$
\alpha_1\ :=\ \alpha\big\vert_{X\setminus Y}\ :\ \mathcal{F}\big\vert_{X\setminus Y}
\ \longrightarrow\ \mathcal{G}\big\vert_{X\setminus Y}
$$
is an isomorphism, then $\alpha$ is an isomorphism.
\end{enumerate}
\end{prp}

For any torsion-free coherent sheaf $\mathcal{E}$ on $X$ and any
Zariski open subset $U\, \subset\, X$, such that the codimension of its complement
$X\setminus U \, \subset\, X$ is at least two, then there is a natural inclusion map
\begin{equation}\label{e1}
\Phi_{\mathcal{E}, U}\ :\ \mathcal{E}\ \hookrightarrow\ \iota_*\iota^* \mathcal{E},
\end{equation}
where $\iota \,:\, U\, \hookrightarrow\, X$ is the inclusion map. A torsion-free coherent sheaf $E$ on $X$
is called \textit{normal} if the map $\Phi_{\mathcal{E}, U}$ in \eqref{e1} is an isomorphism for every
such $U$. Proposition \ref{isomorphism_of_vector_bundles_lemma}(2) can be deduced using the
fact that a reflexive sheaf is normal (cf. \cite[Proposition 1.6]{stable_reflexive_sheaves}).

\section{Endomorphisms of surfaces and non-singular varieties}\label{section first two case}

In this section, we prove the first two cases of \Cref{main_thm_codim_case}. 
An alternative proof of the first case of \Cref{main_thm_codim_case} is discussed in \Cref{non-singular case remark}.

\begin{thm}\label{th1}
Let $X, \, Y$ and $\phi$ be as in \Cref{conv1}, with $\mathrm{codim}_XY\, \geq \, 2$. Assume that at least one of the 
following two statements holds:
\begin{enumerate}[label=(\roman*)]
\item $X$ is non-singular.

\item $X$ is an algebraic surface.
\end{enumerate}
Then $\phi$ is an automorphism.
\end{thm}

\begin{proof}
We continue with the notation of Convention \ref{conv1}.

Assume that $X$ is non-singular. Consider the following exact sequence of K\"{a}hler differentials 
\begin{equation}\label{e3}
\phi^* (\, \Omega_{X}\,)\ \overset{\alpha}{\longrightarrow}\ \Omega_{X}\ \longrightarrow\ \Omega_{X/X}\ \longrightarrow\ 0.
\end{equation}
Note that both $\Omega_{X}$ and $\phi^* (\, \Omega_{X}\,)$ are locally free sheaves. By the hypothesis, the 
restriction of $\alpha$ to $U$ is an isomorphism. So $\alpha$ is an isomorphism which follows from 
\Cref{isomorphism_of_vector_bundles_lemma}. The cokernel $\Omega_{X/X}$ vanishes globally; hence $\phi$ is 
unramified. Consequently, $\phi$ is quasi-finite, and now it is an automorphism
by \Cref{quasi-finite to automorphism}.

Next, assume that $X$ is an algebraic surface. Then we have $\dim Y\, =\, 0$ if $Y$ is nonempty. The restriction 
$$\phi_1\, :\, X \setminus Y \, \longrightarrow X$$ is an open immersion (see Remark
\ref{rem1}). Since $Y$ is a finite set, the map $\phi$ is quasi-finite.
It is an automorphism by \Cref{quasi-finite to automorphism}.
\end{proof}

\section{Submersive morphisms}\label{section four}

We will first prove \Cref{isolated singularity theorem} and then \Cref{Automorphism of non-singular locus theorem}.

\begin{thm}\label{isolated sing section case}
Let $X, \, Y$ and $\phi$ be as in \Cref{conv1}, with $\mathrm{codim}_XY\, \geq \, 2$, such
that $X$ has only finitely many singular points. Denoting the non-singular locus of $X$ by $W$, assume that
the restriction of $\phi$ to $W$ induces an automorphism of $W$. Then $\phi$ is an automorphism.
\end{thm}

\begin{proof}
By the hypothesis, the restriction of $\phi$ to $W$ is quasi-finite. As $X \setminus W$ is finite, it follows 
that $\phi$ is quasi-finite. Now the result follows from \Cref{quasi-finite to automorphism}.
\end{proof}

Next, we prove \Cref{Automorphism of non-singular locus theorem} case by case.

\begin{thm}\label{automorphism restriction result}
Let $X, \, Y$ and $\phi$ be as in \Cref{conv1}, with $\mathrm{codim}_XY\, \geq \, c$. Denote the non-singular
locus of $X$ by $W$. Assume that at least one of the following statements holds:
\begin{enumerate}
\item $c \ = \ 2$, and $\phi$ is surjective.
\item $c \ = \ 2$, and $X$ is a $\mathbb{Q}$-factorial variety.
\item $c \ = \ 2$, and $X$ is locally a complete intersection.
\item $c \ = \ 3$.
\end{enumerate}
Then $\phi\ (W) \ \subseteq \ W$, and the restriction $$\phi\big\vert_{W}\ : \ W \ 
\longrightarrow \ W$$ is an automorphism of $W$.
\end{thm}

The following proposition will be used in the proof of Theorem \ref{automorphism restriction result}.

\begin{prp}\label{open immersion result}
Use Convention \ref{conv1}, and assume that $\mathrm{codim}_XY\, \geq \, 2$. The
restriction 
\begin{align}\label{eqn_open immersion}
\phi_2\ := \
\phi\big\vert_{\phi^{-1}(W)}\ : \ \phi^{-1}(W)\ \longrightarrow \ W
\end{align}
 is an open immersion. In
particular, $\phi^{-1}(W) \cap S \, =\, \emptyset$.
\end{prp}

\begin{proof}
We continue with the notation of Convention \ref{conv1}. The last statement is immediate from
the first one because $S$ (respectively, $W$) is the singular (respectively, non-singular)
locus of $X$. So it suffices to prove the first statement.

In Remark \ref{rem1}, it was already observed that $\phi$ is birational. It follows that $\phi_2$ is also 
birational. The definition of the (minimal) exceptional locus of $\phi_2$ is recalled: A closed subset $E\, 
\subset\, \phi^{-1}(W)$ is called exceptional for $\phi_2$ if it is of the form $\phi_2^{-1}(W \setminus A)$, 
where $A \, \subset\, W$ is some open subset such that the restriction $$\phi_2\big\vert_{\phi_2^{-1}(A)}\ : 
\ \phi_2^{-1}(A)\ \longrightarrow \ A$$ is an isomorphism. Furthermore, $E$ is called minimal exceptional for 
$\phi_2$ if the above $A$ is the unique maximal open subset of $W$ such that the restriction 
$\phi_2\big\vert_{\phi_2^{-1}(A)}$ is an isomorphism. It follows from \Cref{preimage_preserving_lemma} that 
$Y$ is exceptional locus for $\phi$. Denoting $E$ to be the minimal exceptional locus of $\phi_2$, we have
$$
E\ \subset\ \phi^{-1}(W)\cap Y \ \subset\ Y.
$$
This implies that
\begin{equation}\label{e2}
\mathrm{codim}_{\phi^{-1} (W)} 
E\ = \ \mathrm{codim }_X E\ \geq \ \mathrm{codim }_X Y\ \geq\ 2.
\end{equation}
On the other hand, applying Zariski--Van der Waerden purity theorem (cf.
\cite[Proposition 1, Chapter III.9]{Red_book}) to $\phi_2$,
it follows that $$\mathrm{codim}_{\phi^{-1} (W)} E\, = \,1$$
if $E $ is non-empty. But this contradicts \eqref{e2}.

Therefore, we conclude that $E\ = \ \emptyset$, which is the same as the statement that $\phi_2$ is an open immersion.
\end{proof}

\begin{rem}\label{non-singular case remark}
\Cref{open immersion result} gives an alternative proof of \Cref{main_thm_codim_case}(i). In this
case, $W$ is the same as $X$; therefore, we have $\phi^{-1}(W) \ = \ X$. Consequently, $\phi$ is injective, and
hence it is an automorphism, by the theorem of Ax.
\end{rem}

\begin{proof}[{Proof of the first case of \Cref{automorphism restriction result}}]
According to the hypothesis, $\phi$ is surjective. From Proposition \ref{open immersion result}
we know that the restriction $\phi_2$ in \eqref{eqn_open immersion} is an isomorphism. Therefore,
$$\phi_2^{-1}\ : \ W \ \longrightarrow \ W$$ is an injective endomorphism; so $\phi_2^{-1}$ is an automorphism (theorem
of Ax). This yields that $\phi^{-1}(W) \ = \ W$. Thus $\phi_2$ in \eqref{eqn_open immersion} is an automorphism of $W$.
\end{proof}

\subsection{Submersive and differentially complete intersection morphisms}

The rest of the article heavily relies on the terminology and content of \cite{Kallstrom}. Some 
definitions and results from \cite{Kallstrom} are recalled.

\begin{defn}
Take two varieties $X,\, Y$ defined over $k$, and let $f\,:\, X \, \longrightarrow \, Y$ be a regular morphism. The
map $f$ is said to be \textit{differentially a complete intersection} (``d.c.i.'' for short) at a point $x$ of
$X$ if $$\mathrm{pd}_{\mathcal{O}_{X,\, x}} \left( \, \Omega_{X/Y}\, \right)_x\ \leq \ 1.$$ Here
$\mathrm{pd}_R\,M$ denotes the projective dimension of the $R$--module $M$ (same as the minimal length of
free resolutions of $M$). Also, $f$ is said to be d.c.i. if it is so at each point of
$X$ (see \cite[p.~161]{Kallstrom}). A variety $X$ is said to be d.c.i. if the morphism $X\,
\longrightarrow \,\mathrm{spec}\, k$ is so.
\end{defn}

Note that a non-singular variety is differentially a complete intersection.

\begin{defn}\label{ramification locus}
The \textit{ramification locus} of a regular morphism $f\ :\ X \ \longrightarrow \ Y$, denoted
by $B_f$, is defined to be the set of points $x$ of $X$ such that $\left(\Omega_{X/Y}\right)_x$ is not free.
\end{defn}

Given $f\,:\, X \, \longrightarrow \, Y$, we have the following natural sequence of K\"{a}hler differentials
\begin{equation*}\label{eq_temp7}
f^*  (\, \Omega_{Y}\,)\, \longrightarrow \, \Omega_X \, \longrightarrow \, \Omega_{X/Y}\, \longrightarrow 0.
\end{equation*}
Taking dual, we obtain the following exact sequence
\begin{equation}\label{ecf}
0\, \longrightarrow \, \Omega_{X/Y}^{\vee} \, \longrightarrow \, \Omega_X^{\vee} \, \overset{df}{\longrightarrow} \, \left(\,f^*  (\, \Omega_{Y}\,)\, \right)^{\vee}\, \longrightarrow \, \mathcal{C}_f \, \longrightarrow \, 0.
\end{equation}

\begin{defn}\label{de1}
The cokernel $\mathcal{C}_f$ of $df$ is called the critical module. The map $f$ is said to be \textit{submersive}
if $\mathcal{C}_f\,=\,0$ (see \cite[p.~160]{Kallstrom}).
\end{defn}

We now recall Corollary 4.9 of \cite{Kallstrom}. Note that \cite[Corollary 4.9]{Kallstrom}
is in a much more general setting. The following is a weaker version of it.

\begin{lmm}\label{submersive lemma}
Let $X/k$ be a non-singular variety,
and let $Y/k$ be a normal variety. Let $f\,:\, X \, \longrightarrow \, Y$ be a birational morphism.
Furthermore, assume that at least one of the following assumptions holds:
\begin{itemize}
\item $Y$ is locally a complete intersection.
\item $Y$ is $\mathbb{Q}$--factorial.
\end{itemize}
Then the following statements are equivalent:
\begin{enumerate}
\item $f$ is submersive.
\item $f$ is \'{e}tale.
\end{enumerate}
\end{lmm}

\begin{proof}
Assume that $Y/k$ is locally a complete intersection. The projective dimension of the stalks of $\Omega_Y$
is at most one, which follows from \cite[Corollary 9.4]{Kunz_Kahler_differentials}. Thus $Y$ is d.c.i.
Being non-singular, $X$ is also d.c.i. The equivalence of the two statements follows from
\cite[Corollary 4.9(e)]{Kallstrom}.

Assume that $Y$ is $\mathbb{Q}$--factorial. It follows that $Y$ satisfies the condition
$\left(\mathbf{W}\right)$ described in lines 10--11, from below, on page 173 of
\cite{Kallstrom} (see lines 8--9, from below, on  page 173 of \cite{Kallstrom} for the
fact that $\mathbb{Q}$--factorial varieties satisfy condition $\left(\mathbf{W}\right)$). The
equivalence of the two statements now follows from \cite[Corollary 4.9(a)]{Kallstrom}.
\end{proof}

In our situation, we have the following.

\begin{lmm}\label{submersive endomorphism lemma}
Let $X, \, Y$ and $\phi$ be as in \Cref{conv1}, with $\mathrm{codim}_XY\, \geq \, 2$. Then $\phi$ is submersive.
\end{lmm}

\begin{proof}
We continue with the notation of \Cref{conv1}. Consider the first fundamental sequence of K\"{a}hler differentials
\begin{align*}
\phi^*  (\, \Omega_{X}\,) \, \overset{\alpha}{\longrightarrow} \, \Omega_{X} \, \longrightarrow \, \Omega_{X/X}\, \longrightarrow 0.
\end{align*}
Its restriction
$$\phi_1\ := \ \phi\big\vert_U \ : \ U \ \longrightarrow \ V$$ is given to be an isomorphism. This yields
that the restriction 
$$\alpha_1\ := \ \alpha\big\vert_U \ : \ \phi^* (\, \Omega_{X}\,)\big\vert_U \ \longrightarrow \
\Omega_X\big\vert_U$$ is an isomorphism as well. So $\alpha_1^{\vee}$ --- the $\mathcal{O}_U$--dual of
$\alpha_1$ --- is an isomorphism. The homomorphism $\alpha_1^{\vee}$ evidently coincides
with the restriction of $$d\phi \ : \ \Omega_{X}^{\vee} \ \longrightarrow \ (\phi^*(\, \Omega_{X}\,) )^{\vee}$$
to $U$. Since the duals of $\phi^*(\, \Omega_{X}\,)$ and $\Omega_X$ are reflexive (see
\cite[Corollary 1.2]{stable_reflexive_sheaves}), from \Cref{isomorphism_of_vector_bundles_lemma}
it follows that $d\phi$ is an isomorphism. Thus $\mathcal{C}_{\phi}$ vanishes (see \eqref{ecf}). So
$\phi$ is submersive (see Definition \ref{de1}).
\end{proof}

\begin{proof}[{Proof of the second and third cases of \Cref{automorphism restriction result}}]
We continue with the earlier notation and conventions. Consider the restriction map 
\begin{equation*}
\phi_3\ := \ \phi\big\vert_{W} \ : \ W \ \longrightarrow \ X.
\end{equation*} 
By \Cref{submersive endomorphism lemma}, the map $\phi_3$ is submersive. Apply \Cref{submersive lemma}
to conclude that $\phi_3$ is \'{e}tale, in both the cases. In particular, $\phi_3$ is quasi-finite. By
Zariski's Main Theorem, $\phi_3$ is an open immersion. Thus $\phi_3(W)$ (which is the same as $\phi(W)$)
is contained inside $W$. We obtain that $\phi_3$ is an injective
endomorphism of the smooth variety $W$, and hence $\phi_3$ is an automorphism, by the theorem of Ax. 
\end{proof}

We will now prove the final part of Theorem \ref{automorphism restriction result}.

\begin{proof}[{Proof of the fourth case of \Cref{automorphism restriction result}}]
We continue with the notation of \Cref{conv1}. In this case, $c$ is assumed to be $3$. That is, $\phi$ is 
injective at all points of codimensions $0$, $1$ and $2$. Consequently, the map $\phi_x$ in 
\eqref{stalk_map} is an isomorphism for all points $x$ of codimension up to $2$.

Consider the restriction map 
\begin{equation*}
\phi_3\ := \ \phi\big\vert_{W} \ : \ W \ \longrightarrow \ X.
\end{equation*}
Our goal is to apply \cite[Theorem 4.6]{Kallstrom} (with $S\ = \ \rm{spec} \ k$) to $\phi_3$. The
hypotheses of \cite[Theorem 4.6]{Kallstrom} are checked one by one. Since $\phi_3$ is a birational morphism of 
normal varieties defined over $k$ (which is algebraically closed of characteristic zero), it satisfies the
condition $\left(\mathbf{F}\right)$ of \cite[Theorem 4.6]{Kallstrom} (condition $\left(\mathbf{F}\right)$ is
described in lines 5--6, from below, on page 173 of \cite{Kallstrom}). As $W$ is non-singular, the ramification 
locus $B_{X/k}$ (see \Cref{ramification locus}) is empty. The first condition of \cite[Theorem 4.6]{Kallstrom} is
satisfied by definition and the convention that $\mathrm{codim}^-_X \ \emptyset \ = \ \infty$. For more details about the notation, see
\cite[p. 159]{Kallstrom}.

For the second condition of \cite[Theorem 4.6]{Kallstrom}, observe that the stalk map $\left(\phi_3\right)_x$ 
in \eqref{stalk_map} is an isomorphism for all point $x$ of height up to $2$ (as $c$ is given to be $3$). 
Thus $\phi_3(x)$ is a regular point of $X$ for such a point $x$. Consequently, for any point $x$ of height at 
most two, $\Omega_{X, \phi_3(x)}$ is a free module. Therefore, the second condition is automatically 
satisfied.

The third condition of \cite[Theorem 4.6]{Kallstrom} is more subtle. Consider the following exact sequence of 
K\"{a}hler differentials
\begin{equation}\label{exact sequence of differentials}
\phi_3^*(\, \Omega_{X}\,) \, \overset{\alpha}{\longrightarrow} \, \Omega_W \, \longrightarrow \, \Omega_{W/X}\, \longrightarrow 0.
\end{equation}
Following the notation of \cite{Kallstrom} (see (1.2) and (1.3) of page 159), \eqref{exact sequence of differentials} splits into two short exact sequences
\begin{align}
0 \ \longrightarrow \ \Gamma \ \longrightarrow \phi_3^*(\, \Omega_{X}\,)\, \longrightarrow \, \mathcal{V} \ \longrightarrow \ 0, \label{ses_temp}\\
0 \, \longrightarrow \, \mathcal{V} \ \longrightarrow \ \Omega_W \, \longrightarrow \, \Omega_{W/X}\, \longrightarrow 0,
\end{align}
where $\Gamma\,=\,\text{kernel}(\alpha)$ and ${\mathcal V}\,=\, \text{Image}(\alpha)$.
As $W$ is irreducible and non-singular, the generic point is the only associated point of $W$, and $\phi_3$ is smooth at that
point (because $\phi_3$ is birational). By \cite[Proposition 2.13]{Kallstrom}, we have $\Gamma\, =\, 0$.
Therefore, from \eqref{ses_temp} it follows that $\mathcal{V} \ = \ \phi_3^*(\, \Omega_{X}\,)$; also,
\eqref{exact sequence of differentials} modifies to the following short exact sequence
\begin{align}\label{ses_differentials}
0 \ \longrightarrow \ \phi_3^*(\, \Omega_{X}\,)\, \overset{\alpha}{\longrightarrow} \,
\Omega_W \, \longrightarrow \, \Omega_{W/X}\, \longrightarrow 0.
\end{align}
As $\Gamma\,=\, 0$, in particular $\Gamma$ is locally free, from the second sentence
of Remark 4.7 on page 175 of \cite{Kallstrom}
we conclude that the third condition of \cite[Theorem 4.6]{Kallstrom} is satisfied.

It follows from \cite[Theorem 4.6]{Kallstrom} that the codimension of the irreducible components
of $B_{\phi_3}$ are at most $1$ (the definition of $\rm{codim}^+_X \
B_{\phi_3}$ is given in \cite[p. 159]{Kallstrom}). On the other hand, the stalk map $\alpha_x$, in \eqref{ses_differentials}, is an
isomorphism for all points $x$ of codimension up to $2$. Hence we have
$$\mathrm{codim} \ \mathrm{supp} \ \left(\Omega_{W/X}\right)\ \geq \ 3.$$ So, $B_{\phi_3}$ is
supported on points of codimension at least $3$. This yields that
$$B_{\phi_3} \ = \ \emptyset.$$
Therefore $\Omega_{W/X}$ is a locally free sheaf of constant rank, and the rank is zero, because
it was shown to be zero on the points of $U \cap W$. Thus $\phi_3$ --- being unramified --- is quasi-finite.
As $\phi_3$ is birational, it is an open immersion, by Zariski's
Main Theorem. Thus $\mathrm{Image}\, (\phi_3)\ \subseteq\ W$, and
$$\phi_3\ : \ W \ \longrightarrow \ W$$ is an injective map, and hence $\phi_3$ is an automorphism.
\end{proof}

\section{Endomorphisms of locally complete intersection varieties}\label{section fourth case}

In this section, the remaining part of \Cref{main_thm_codim_case} will be proved. This is based on homological approach. For the discussions of homological concepts such as depth, homological dimensions, Serre's conditions $(\mathrm{R}_k)$ and $(\mathrm{S}_k)$, we refer to the standard textbook \cite{Bruns-Herzog}. The projective dimension of a module $M$ over a local ring $(R,\mathfrak{m},k)$ is denoted by $\mathrm{pd}_R\, M$, and the $\mathfrak{m}$-depth of $M$ is denoted by $\mathrm{depth}\ M$.

\begin{lmm}\label{projective dimension lemma}
Let $R, \, S$ be two local domains, and let $f\, :\, R \, \longrightarrow \, S$ be an injective
homomorphism of local rings. Given a $R$-module $M$, if $\mathrm{pd}_R\,M\, \leq \, 1$, then
$$\mathrm{pd}_S\,M \otimes_RS\ \leq \ 1.$$
\end{lmm}

\begin{proof}
If $\mathrm{pd}_R\,M\, = \, 0$, then $\mathrm{pd}_S\,M \otimes_RS\,=\, 0$.
Assume that $\mathrm{pd}_R\,M\, = \, 1$.
Let 
\begin{equation*}
0 \, \longrightarrow \, F_1 \, \longrightarrow \, F_0 \, \longrightarrow \, M \, \longrightarrow\, 0
\end{equation*}
be a minimal free resolution of $M$. Tensoring it with $S$ yields the following exact sequence of $R$-modules
\begin{equation*}
0\, \longrightarrow \,\mathrm{Tor}^R_1(S,\,M) \, \longrightarrow \, F_1\otimes_RS \, \longrightarrow \,
F_0\otimes_RS \, \longrightarrow \, M \otimes_R S \, \longrightarrow\, 0.
\end{equation*}
Since $R$ is assumed to be a domain, $\mathrm{Tor}^R_1(S,\,M)$ is a torsion $R$-module. On the other hand,
as $f$ is injective, $F_1\otimes_RS$ --- being free as an $S$--module --- is a torsion-free $R$--module.
Thus $\mathrm{Tor}^R_1(S,\,M)\, = \,0$. The lemma follows. 
\end{proof}

The following result will be helpful for our purpose.

\begin{lmm}\label{reflexivity lemma}
Let $X$ be an algebraic variety defined over $k$ such that $X$ is locally a complete intersection.
Let $\mathcal{F}$ be a coherent sheaf on $X$ with
$$\mathrm{pd}_{\mathcal{O}_{X,x}} \mathcal{F}_x\ \leq \ 1$$ for each point $x$ of $X$.
Also, assume that $\mathcal{F}_x$ is free for all points $x$ of $X$ of codimension at most $c$.
Then $\mathcal{F}$ is torsion-free if $c \ = \ 1$, and it is reflexive if $c \ = \ 2$.
\end{lmm}

\begin{proof}
Since $X$ is locally a complete intersection, for each point $x$ of $X$ the depth and the dimension of $\mathcal{O}_{X,x}$ coincide. 

Assume that $c \ = \ 1$. Let $x$ be a point of codimension at least $2$. Applying the Auslander--Buchsbaum 
formula (cf. \cite[Theorem 1.3.3]{Bruns-Herzog}), it follows that $$\mathrm{depth}\ \mathcal{F}_x \ \geq 1.$$ Also the depth of $\mathcal{F}_x$ is $1$ 
for all codimension $1$ points $x$, and $\mathrm{depth}\ \mathcal{F}_x \ = \ 0$ if and only if $x$ is the 
generic point of $X$. In other words, the generic point of $X$ is the only associated point of $\mathcal{F}_x$, and $\mathcal{F}_x$ has the property $(\mathrm{S}_1)$ (cf. \cite[p. 63]{Bruns-Herzog}). 
Now we apply \cite[\href{https://stacks.math.columbia.edu/tag/0AUV}{Tag 0AUV}]{stacks-project} to 
$\mathcal{F}_x$, and conclude that it is torsion-free for each point $x$ of $X$. Thus $\mathcal{F}$ is 
torsion-free.

Next, assume that $c\ =\ 2$. Let $x$ be a point of $X$ such that $\mathrm{depth}\, \mathcal{O}_{X,x}\, = \ 2$.
Then $\mathcal{F}_x$ is free; hence we have $$\mathrm{depth}\ \mathcal{F}_x\ = \ 2$$ by the Auslander--Buchsbaum
formula. Next, if $$\mathrm{depth}\, \mathcal{O}_{X,x}\ \geq \ 3,$$ then applying the Auslander--Buchsbaum
formula,
$$\mathrm{depth}\ \mathcal{F}_x\ \geq \ 2.$$ Now the reflexivity of $\mathcal{F}$ follows from \cite[Proposition 
1.3]{stable_reflexive_sheaves}.
\end{proof}

We are now ready to prove the main theorem of this section.

\begin{thm}\label{th3}
Let $X,\, Y$ and $\phi$ be as in \Cref{main_thm_codim_case}. Furthermore, assume that $X$ is locally a complete
intersection satisfying the condition $\left(\mathrm{R}_2\right)$. Then $\phi$ is an automorphism.
\end{thm}

\begin{proof}
We continue with the notation of \Cref{conv1}. As $X$ satisfies $\left(R_2\right)$ (cf. \cite[p. 71]{Bruns-Herzog}), 
it follows from \cite[Proposition 8.1]{Lipman} that $\Omega_X$ is reflexive. Given a point $x$ of
$X$, the stalk $\left( \phi^*(\, \Omega_{X}\,) \right)_x$ is the same as 
$\Omega_{X ,\,\phi\,(x)} \otimes_{\mathcal{O}_{X,\,\phi\,(x)}} \mathcal{O}_{X,\,x}$. It is well known
that 
$$\mathrm{pd}_{\mathcal{O}_{X,\, \phi\,(x)}} \Omega_{X ,\,\phi\,(x)}\ \leq \ 1$$
(cf. \cite[Corollary 9.4]{Kunz_Kahler_differentials}).
Also, applying the combination of \Cref{projective dimension lemma} and \Cref{injectivity of stalk maps}
it follows that $$\mathrm{pd}_{\mathcal{O}_{X,\, x}} \left( \phi^*(\, \Omega_{X}\,) \right)_x \ \leq \ 1. $$

Consider the exact sequence of K\"{a}hler differentials 
\begin{equation}\label{eqn_temp}
\phi^*(\, \Omega_{X}\,) \, \overset{\alpha}{\longrightarrow} \, \Omega_{X} \, \longrightarrow \,
\Omega_{X/X}\, \longrightarrow\, 0.
\end{equation} 
We will now show that $\alpha$ in \eqref{eqn_temp} is an isomorphism. The restriction
$$ \phi\big\vert_{W}\ :\ W \
\longrightarrow \ W$$ was shown to be an automorphism in \Cref{Automorphism of non-singular locus theorem}.
Therefore, the restriction $$\alpha\big\vert_{W}\ : \ \phi^*(\, \Omega_{X}\,)\big\vert_{W} \ \longrightarrow \
\Omega_{X}\big\vert_{W}$$ is an isomorphism. As $X$ is regular in codimension $2$, it follows that
$(\phi^*(\, \Omega_{X}\,) )_x$ is free for each point $x$ of codimension at most $2$. Reflexivity of
$\phi^* (\, \Omega_{X}\,)$ follows from \Cref{reflexivity lemma}. Since $\alpha\big\vert_{W}$ is an
isomorphism, we have $\alpha$ to be an isomorphism by \Cref{isomorphism_of_vector_bundles_lemma}.

The exact sequence in \eqref{eqn_temp} ensures that $\Omega_{X/X}\ =\ 0$. Thus $\phi$ is quasi-finite.
Now $\phi$ is an automorphism by \Cref{quasi-finite to automorphism}.
\end{proof}

\section*{Acknowledgement}
We are very grateful to the referee for helpful comments. We are very thankful to Utsav Choudhury, R. V. 
Gurjar, Vivek Sadhu, and Chandranandan Gangopadhay for several fruitful discussions.

The first author is partially supported by a J. C. Bose Fellowship (JBR/2023/000003). The second author is supported by the
INSPIRE faculty fellowship (IFA21-MA 161) funded by the DST, 
Govt. of India. Both authors would like to thank NISER Bhubaneswar, where this work was initiated,
for its hospitality during the Geometry, Analysis, Mathematical Physics (GAMP- 2023) workshop.

\bibliographystyle{acm}

\end{document}